\documentclass[hidelinks,12pt,reqno]{amsart}
\usepackage{tikz,dsfont}
\usepackage{amsmath, amssymb, graphicx, mathrsfs, hyperref}
\usepackage{verbatim}
\usepackage{thmtools}
\usepackage{thm-restate}
\usepackage{enumerate}
\newtheoremstyle{dotless}{}{}{\itshape}{}{\bfseries}{}{ }{} 
\theoremstyle{dotless}

\newtheorem{theorem}{Theorem}[section]
\newtheorem{lemma}[theorem]{Lemma}

\newtheorem{prop}[theorem]{Proposition}
\newtheorem{conj}[theorem]{Conjecture}

\newtheorem{defn}[theorem]{Definition}
\newtheorem{remark}[theorem]{Remark}

\newtheorem{exam}[theorem]{Example}

\usepackage{cleveref}

\begin{document} 

\pagenumbering{arabic} \setcounter{page}{1}

\title{On higher energy decompositions and the sum--product phenomenon }

\author{George Shakan}

 \thanks{The author is partially supported by NSF grant  DMS--1501982 and would
like to thank Kevin Ford for financial support. The author also thanks Kevin Ford and Oliver Roche--Newton for useful comments and suggestions, as well as the referee for a meticulous and timely reading and helpful suggestions.}
\address{Department of Mathematics \\
University of Illinois \\
Urbana, IL 61801, U.S.A.}
\email{shakan2@illinois.edu}

\begin{abstract}
Let $A \subset \mathbb{R}$ be finite. We quantitatively improve the Balog--Wooley decomposition, that is $A$ can be partitioned into sets $B$ and $C$ such that $$\max\{E^+(B) , E^{\times}(C)\} \lesssim |A|^{3 - 7/26}, \ \ \max \{E^+(B,A) , E^{\times}(C, A) \}\lesssim |A|^{3 - 1/4}.$$ We use similar decompositions to improve upon various sum--product estimates. For instance, we show
$$ |A+A| + |A A|  \gtrsim   |A|^{4/3 + 5/5277}.$$
\end{abstract}

\maketitle
\tableofcontents

\section{Introduction}

Let $A, B \subset \mathbb{R}$ be finite. We define the {\em sumset} and {\em product set} via $$A+B := \{a + b : a \in A , \ b \in B\} , \ \ \ \ \ \ AB := \{ab : a \in A , \ b \in B\}.$$ In this paper, we say $b \gtrsim a$ if $a = O(b \log^c |A|)$ for some $c >0$ and $a \sim b$ if $b \gtrsim a$ and $a \gtrsim b$.
Equipped with these definitions we are ready to state the Erd\H{o}s--Szemer\'{e}di sum--product conjecture.

\begin{conj}\label{esconj}\cite{ES} Fix $\delta \leq 1$. Then for any finite $A \subset \mathbb{Z}$, one has
 $$|A+A| + |A A| \gtrsim|A|^{1 + \delta}.$$
\end{conj}

In the same paper, Erd\H{o}s and Szemer\'{e}di showed that Conjecture \ref{esconj} holds for some $\delta >0$, which began the history of the so called ``sum--product conjecture." Fourteen years passed until Nathanson \cite{Na} modified their proof and made the first quantitative estimate, showing Conjecture \ref{esconj} holds for $\delta = 1/31$. Ford \cite{Fo} quickly improved  Nathanson's argument to obtain $\delta = 1/15$ is admissible in Conjecture \ref{esconj}. Ford did not have this world record for long, as within months Elekes \cite{El} showed Conjecture \ref{esconj} holds for $\delta = 1/4$. Elekes' techniques were completely different, as he remarkably made use of the Szemer\'{e}di--Trotter theorem from incidence geometry. His work marks the beginning of modern progress towards resolving Conjecture \ref{esconj}.

Solymosi \cite{So2} showed $\delta = 3/11$ is admissible in Conjecture \ref{esconj}. Later, in \cite{So} he used elementary geometry in a clever way to improve this to $\delta = 1/3$. This remained the world record for six years, until Konyagin and Shkredov \cite{KS} combined Solymosi's argument with Shkredov's work in additive combinatorics \cite{Sh, Shadd} to increase Solymosi's exponent. In a more recent paper  \cite{KS2}, the same authors proved  $\delta = 4/3 + 5/9813$ is admissible in Conjecture \ref{esconj}. Rudnev, Shkredov and Stevens \cite{RSS} replaced a ``few sums many products" lemma used in \cite{KS2} to obtain the world record that Conjecture \ref{esconj} holds for $\delta = 4/3 + 1/1509$. We make further improvements to show the following.

\begin{theorem}\label{sumprod} Let $A \subset \mathbb{R}$ be finite. Then  $$|AA| + |A+A| \gtrsim |A|^{4/3 + 5/5277}.$$
\end{theorem}

Thus we improve Solymosi's exponent by nearly $\frac{1}{1000}$. We remark that Theorem \ref{sumprod} also holds if one replaces the product set with the quotient set.

We now turn our attention to decomposition results, which is the main motivation for the current work. About 35 years after the original sum--product conjecture, Balog and Wooley \cite{BW} provided a new way of looking at the problem of intrinsic interest and applicable (see \cite{Ha} for the first application). To state their results, we recall some definitions. Again, let $A, B \subset \mathbb{R}$ be finite. We define two representation functions of $x \in \mathbb{R}$: $$r_{A-B}(x) = \#\{(a,b) \in A \times  B : x = a - b\} , \ \ \ \ \ r_{A/B}(x) =  \#\{(a,b) \in A \times  B : a = xb\} .$$ The {\em additive energy} and {\em multiplicative energy} of $A$ and $B$ are defined via 

$$E^{+}(A,B) = \sum_x r_{A-B}(x)^2 , \ \ \ \ \ \ \ \ E^{\times}(A,B) = \sum_x r_{A/B}(x)^2.$$ We set $E^{+}(A) = E^{+}(A,A) $ and $E^{\times}(A) = E^{\times}(A,A) $. Heuristically, $E^+(A)$ is large when $A$ has additive structure. This is seen more clearly by the relation $$E^+(A) = \#\{(a , b , c, ,d) \in A^4 : a +b = c+ d\}.$$

\begin{theorem}\cite{BW}\label{bw}
Let $A$ be a finite subset of the real numbers and $\delta = 2/33$. Then there exist $B,C$ that partition $A$ satisfying $${\rm max}\{ E^{+}(B) , E^{\times}(C) \} \lesssim |A|^{3-\delta}, \ \ \ \ {\rm max} \{ E^{+}(B,C) , E^{\times}(B,C) \}\lesssim |A|^{3-\delta/2}.$$
\end{theorem}

Thus any set may be decomposed into two sets, one with little additive structure and one with little multiplicative structure. Note that Theorem \ref{bw} with exponent $\delta$ implies Conjecture \ref{esconj} with exponent $\delta$, via Cauchy--Schwarz:

$$|A|^2 |B|^2 \leq |A+B| E^{+}(A, B) , \ \ \ \ \ |A|^2 |B|^2 \leq |AB| E^{\times}(A, B) .$$ This is the so called ``energy analog" of the sum--product problem. In the same paper Balog and Wooley provided the example \begin{equation}\label{ex} \{ (2m-1) 2^j : 1 \leq m \leq S, 1 \leq j \leq P\},\end{equation} which shows, when $S=P^2$, it is not possible to improve Theorem \ref{bw} beyond $\delta = 2/3$. 

Balog and Wooley use an iterative argument to combine two key lemmas and prove Theorem \ref{bw}.  The first is a rather easy lemma concerning how the multiplicative energy behaves with respect to unions. The second is at the heart of the proof, which says if the additive energy is large, then there is a large subset that has small multiplicative energy. To accomplish this, they utilized Solymosi's \cite{So} sum--product result as well as the Balog--Szemer\'{e}di--Gowers theorem from additive combinatorics. Konyagin and Shkredov \cite{KS2} replaced this lemma with a completely different lemma of their own that allowed them to show $\delta = 1/5$ is admissible in Theorem \ref{bw}. This lemma is what inspired the current work. Finally Rudnev, Shkredov, and Stevens \cite{RSS} improved this to $\delta = 1/4$, which is the energy analog of Elekes' result towards Conjecture \ref{esconj}. We improve this to $\delta = 7/26$ below. To fully state our contribution, we require a few more definitions.

We now introduce the third order energies of a set: $$ E_3^{+}(A,B) = \sum_x r_{A-B}(x)^3 , \ \ \ \ \ \ \ \ E_3^{\times}(A,B) = \sum_x r_{A/B}(x)^3.$$ We set $E_3^{+}(A) = E_3^{+}(A,A) $ and $E_3^{\times}(A) = E_3^{\times}(A,A) $.

We first provide motivation for working with higher moments. It starts with the Szemer\'edi--Trotter theorem, which has played a pivotal role in the sum--product problem since \cite{El} (see chapter 8 of \cite{TV}).

\begin{theorem}\label{ST}[Szemer\'edi--Trotter] Let $P$ be a finite set of points and $L$ be a finite set of lines. Then the number of incidences between $P$ and $L$ is bounded from above: $$\#\{(p,\ell) \in P \times L : p \in \ell\} \leq 4 |P|^{2/3} |L|^{2/3} + 4 |P| + |L|.$$
\end{theorem}

Elekes' result can be recovered by applying Theorem \ref{ST} to $P = (A+A) \times AA$ and $L =\{y = a(x -c) : a, c \in A\}$. Now, for an arbitrary point set, $P$, if we apply Szemer\'edi--Trotter to the set of $\Delta$--popular lines and $P$, we can simplify to obtain $$|L| \lesssim \max \{|P|^2 \Delta^{-3}, |P| \Delta^{-1} \}.$$ Typically the first term is larger (for instance if $P = A \times A$), and so we see that Szemer\'edi--Trotter is most naturally a {\em third} moment estimate.

At the forefront of a number of works concerning the sum--product phenomenon, i.e. \cite{KS, KS2, RSS, MRS, SS, Shdstar}, is the quantity $d^{+}(A)$.

\begin{defn}\label{dplus}
Let $A \subset \mathbb{R}$ finite. We define  $$d^{+}(A) := \sup_{B \neq \emptyset}\frac{ E_3^+(A,B) }{|A| |B|^2},$$ and the multiplicative analog  $$d^{\times}(A) := \sup_{B \neq \emptyset}\frac{ E_3^{\times}(A,B) }{|A| |B|^2}.$$
\end{defn} 

It follows that $1 \leq d^{+}(A), d^{\times}(A) \leq |A|$. Intuitively, the closer $d^{+}(A)$ is to $|A|$, the more additive structure $A$ has and the closer $d^{\times}(A)$ is to $|A|$, the more multiplicative structure $A$ has. Observe that the supremums in Definition \ref{dplus} are achieved for some $|B| \leq |A|^2$ since $$\frac{E_3^+(A,B) }{|A||B|^2} \leq \frac{|A|^2}{|B|}, \ \ \  \frac{E_3^{\times}(A,B) }{|A||B|^2} \leq \frac{|A|^2}{|B|}.$$

\begin{remark}\label{op} We have that $d^+(A) \sim \widetilde{d}^+(A)$, where $\widetilde{d}^+(A)$ is the smallest quantity such that $$\#\{x : r_{A - B}(x) \geq \tau \} \leq \widetilde{d}^+(A) |A| |B|^2 \tau^{-3},$$ holds for all finite $B \subset \mathbb{R}$ and $\tau \geq 1$ \cite[Lemma 17]{Shdstar}.  Indeed by Chebyshev's inequality, $\widetilde{d}^+(A) \leq d^+(A)$ since $$\#\{x : r_{A - B}(x) \geq \tau \} \leq \tau^{-3} \sum_x r_{A - B}(x)^3 \leq \frac{E_3^+(A,B)}{|A||B|^2} |A||B|^2 \tau^{-3}.$$ The reverse inequality follows, up to a logarithm, from a dyadic decomposition $$\frac{E_3^+(A,B)}{|A||B|^2} \sim \frac{1}{|A||B|^2} \displaystyle \max_{1 \leq \tau \leq |A|}  \#\{x : \tau \leq r_{A-B}(x) < 2 \tau\} \tau^3 \leq \widetilde{d}^+(A) .$$
In previous literature, $\widetilde{d}^+(A)$ been taken as the definition of $d^+(A)$. Finally, we remark that  $d^+(A)$ is related to the following operator norm 
$$d^+(A) = \frac{1}{|A|}||T_A||_{\ell^{3/2} \to \ell^3}^3, \ \  \ T_A(f) :=  \sum_{x} f(x) 1_A(y+x).$$ Thus the quantity $d^+(A)$ arises from thinking of $A$ as an {\em operator} rather than a {\em set}.
\end{remark}

The quantities $d^{+}(A)$ and $d^{\times}(A)$ can be thought of as a $\ell^3$ estimate for $r_{A-B}$ and $r_{A/B}$, where we are allowed to vary $B$. This flexibility in choosing $B$ has proved useful in applications. 

\begin{exam} Consider a random $A \subset \{1 , \ldots , n\}$ where each element is chosen independently and uniformly with probability $p > n^{-1/3}$. Clearly $|A| \sim pn$ with high probability. Let $B = \{1 , \ldots , n\}$. It follows from Chernoff's inequality (for instance, Chapter 1 of \cite{TV}) and the union bound, that every $x$ with $r_{B-B}(x) \geq  p^{-2} \log n$ satisfies
$$r_{A-B}(x) \sim p \cdot r_{B-B}(x) , \ \ \ \  r_{A-A}(x) \sim p^2 \cdot  r_{B-B}(x) .$$ 
This quickly implies $$\frac{ E_3^+(A,B) }{|A| |B|^2} \sim p|A|, \ \  \frac{ E_3^+(A) }{|A|^3} \sim p^2|A| , \ \ \frac{ E^+(A,B) }{|A| |B|} \sim |A|, \ \ \frac{ E^+(A) }{|A|^2} \sim p|A| .$$ In this example $d^+(A)$ is larger than what is predicted by the third order energy, where we only allow $B=A$ in Definition \ref{dplus}. Furthermore, the analog of Definition \ref{dplus} for additive energy is as large as possible, that is $\gtrsim |A|$; however, using more involved techniques one can show $d^+(A) \sim p |A|$. 
Thus the trivial bounds $$\frac{E_3^+(A)}{|A|^3} \leq d^+(A)  \leq \max_{B \neq \emptyset} \frac{ E^+(A,B) }{|A| |B|},$$ are not tight in general.
\end{exam}

Sumset and product set information can be deduced from upper bounds for $d^+(A)$ and $d^{\times}(A)$, respectively. For instance, a simple application of Cauchy--Schwarz and Definition \ref{dplus} applied to $B = A$ reveal \begin{equation}\label{triv} \frac{|A|^4}{|A+A|} \leq  E^+(A) \leq  E_3^{+}(A)^{1/2} \left( \sum_x r_{A-A}(x) \right)^{1/2} \lesssim  d^+(A)^{1/2} |A|^{5/2}.\end{equation} Thus we find that using this argument, one can only show $|A+A| \gtrsim |A|^{3/2}$, even with optimal information for $d^+(A)$. Improvements have been made to \eqref{triv}, which highlights the advantage of allowing $B$ to vary in Definition \ref{dplus}. 

\begin{theorem}\label{sum} [\cite[Theorem 11]{Sh} , \cite[Corollary 10]{KR}, see also \cite{SS}, \cite[Theorem 13]{MRS}, see also \cite{Shadd}] Let $A \subset \mathbb{R}$. Then $$|A+A| \gtrsim |A|^{58/37}  d^+(A)^{-21/37} , \ \ \ \ |A-A| \gtrsim |A|^{8/5}  d^+(A)^{-3/5}, \  \ \ \ E^{+}(A) \lesssim d^{+}(A)^{\frac{7}{13}} |A|^{\frac{32}{13}}.$$ 
\end{theorem}

The multiplicative versions of these bounds all hold by applying the additive version to the bigger of $\log \{a \in A : a > 0\}$ and $\log -\{a \in A : a< 0\}$. Thus one basic strategy in several sum--product improvements is as follows: use Szemer\'edi--Trotter to obtain a third moment estimate and then use Theorem \ref{sum} to get improved sum--product bounds. The strength of such theorems can be accurately tested by setting $d^{+}(A) = 1$. Thus the result for difference sets is slightly stronger than that of sumsets and both are stronger than what we can say about the more general additive energy. Quantitative improvements to Theorem \ref{sum} would improve all of our main theorems.

\begin{remark}\label{con} Suppose that one was able to improve upon Theorem \ref{sum} to $$|A+A| \gtrsim |A|^2 d^+(A)^{-1}.$$ Then, combining this with the multiplicative version: $$|AA| \gtrsim |A|^2 d^{\times}(A)^{-1},$$ and applying Theorem \ref{main} below, we would have the sum--product estimate $$|A+A| |AA| \gtrsim |A|^{3}.$$ With this viewpoint, the obstacle to further improvements of Conjecture \ref{esconj} is our current individual understanding of ``sum" and ``product," rather than the combination of the two. The question is how much second moment information can be extracted from third moment information.
\end{remark}

Information about $d^{+}(A)$ and $d^{\times}(A)$  can be used in conjunction with Theorem \ref{sum} to obtain bounds for sum--product problems. Our next theorem shows that these two quantities are related. 

\begin{theorem} \label{main} Let $A \subset \mathbb{R}$ be finite. Then there exists $X , Y \subset A$ such that
\begin{itemize}
\item[(i)] $X \cup Y = A$,
\item[(ii)] $|X| , |Y| \geq |A| /2$,
\item[(iii)] $d^+(X) d^{\times}(Y) \lesssim |A|$.
\end{itemize}
\end{theorem}

This is optimal, as can be seen by taking $A$ to be an arithmetic or geometric progression. We point out that the sets $X$ and $Y$ in Theorem \ref{main} have the convenient property that they are both of size at least $|A|/2$, which has not always been the case with decomposition results; this type of result first appeared in \cite[Theorem 12]{RSS}. Theorem \ref{main} can be interpreted as a $d^+, d^{\times}$ analog of Elekes' \cite{El} sum--product bound, in light of \eqref{triv}. Since Theorem \ref{sum} is better than \eqref{triv}, we can go beyond this Elekes threshold, answering a question in \cite{RSS}.

\begin{restatable}{theorem}{decomp}\label{decomp}
Let $A \subset \mathbb{R}$ be finite and $\delta = 1/4$. Then there exist $B,C$ that partition $A$ with $${\rm max}\{d^{+}(B) , d^{\times}(C)\} \lesssim 	 |A|^{1 - 2 \delta}.$$ Furthermore, $${\rm max}\{ E^{+}(B) , E^{\times}(C) \} \lesssim |A|^{3-\frac{14}{13} \delta} , \ \ \ {\rm max}\{ E^{+}(B,A) , E^{\times}(C, A) \} \lesssim |A|^{3-\delta}.$$
\end{restatable}

This improves upon \cite[Theorem 4]{Shdstar} as well as \cite[Theorem 1]{RSS} and builds upon the work found there. Note that in the last inequality we have a $\delta$ in place of a $\delta / 2$. While Theorem \ref{main} is optimal, we do not expect Theorem \ref{decomp} to be optimal. This lies at the heart of the sum--product phenomenon. With current technology, we are unable to fully rule out the possibility of a set with partial additive and multiplicative structure.  Note the example above in \eqref{ex} shows that one cannot prove $2 \delta > 3/4$, as explained in \cite{Shdstar}.

We now mention more applications of Theorem \ref{main}. First, we consider the difference--quotient and difference--product problems. For $A \subset \mathbb{R}$, we set $$A^{-1} = \{a^{-1} : a \in A\},$$ where we adopt the convention that $0^{-1} = 0$. 
\begin{conj}\label{diff} Let $\delta \leq 1$. Then for any finite $A \subset \mathbb{R}$, one has $$ |A-A| + |A A^{-1}| \gtrsim |A|^{1 + \delta},$$ $$|A-A| + |AA| \gtrsim |A|^{1 + \delta}.$$
\end{conj}
Solymosi's \cite{So} techniques do not work for difference sets, but Elekes' \cite{El} do which shows that Conjecture \ref{diff} holds for $\delta = 1/4$. Solymosi's earlier work \cite{So2} implies that $\delta = 3/11$ is admissible in Conjecture \ref{diff}. Konyagin and Rudnev \cite{KR} adapted techniques from \cite{SS} to show the first statement of Conjecture \ref{diff} holds for $\delta = 1 + 9/31$ and the second statement holds for $\delta = 1 + 11/39$. Using Theorem \ref{main}, we improve their results.
\begin{theorem}\label{diffquot} Let $A \subset \mathbb{R}$ be finite. Then $$|A-A| + |A A^{-1}| \gtrsim   |A|^{1 + 3/10}, $$
$$|A-A| + |A A| \gtrsim   |A|^{1 + 7/24}$$
\end{theorem}
In a similar spirit to the sum--product phenomenon, there are a host of ``expander problems," for instance \cite{BR, RN,MRS, Shkredov, RRSS}. They roughly state that when one creates a set by combining addition and multiplication, the resulting set should be large. For instance, it is of interest to find the best lower bounds for $$|AA\pm A|, \  |AA\pm AA|, \  |A(A\pm A)|, \ \max_{a \in A} |A(A \pm a)|.$$ Typically what happens is that one can apply Szemer\'{e}di--Trotter to obtain a lower bound of the order of magnitude of $|A|^{3/2}$ (see chapter 8 of \cite{TV}) and improving upon this takes additional ideas that usually depend of the structure of the expander (see for instance \cite{RN, MRS, RRSS}). The problem of $AA+A$ is unique in that it has resisted improvements from Szemer\'{e}di--Trotter (see \cite{Shkredov}), until a very recent preprint of Roche--Newton, Ruzsa, Shen, and Shkredov \cite{RRSS}. Typically expanders are conjectured to have size $\gtrsim |A|^2$, but we are usually far from proving so. 

We use Theorem \ref{main} to improve upon the lower bound for the expanders found in \cite{MRS}. Our idea is to use their techniques to the subsets of $A$ appearing in Theorem \ref{main} which have more suitable additive and multiplicative structure. 

\begin{theorem}\label{expander}
Let $A \subset \mathbb{R}$ be finite. Then $$|A(A-A)| \gtrsim |A|^{3/2 + 7/226},$$ $$|A(A+A)| \gtrsim |A|^{3/2 + 1/46},$$ $$\max_{a\in A} |A(A \pm a)| \gtrsim |A|^{3/2 + 1/182}.$$
\end{theorem}

 Note that Solymosi's technique in \cite{So} is better suited for sumsets, while Shkredov's and his coauthors techniques (as in Theorem \ref{sum}) are better suited for difference sets. This subtlety is not at the heart of the sum--product phenomenon, so we mention the following theorem which we will prove during the proof of Theorem \ref{sumprod}.

\begin{theorem}\label{random}
Let $A \subset \mathbb{R}$ be finite. Then $$|A+A| + |A-A| + |AA| + |AA^{-1}| \gtrsim |A|^{4/3 + 1/753}.$$
\end{theorem}

The work in this paper builds directly upon the works in \cite{BW, El,KR, MRS1, RSS, SS, Sh, Shdstar, Shadd, So}. It is worth noting that there are orthogonal works addressing the sum--product phenomenon. Chang \cite{Ch} and Bourgain and Chang \cite{BC} have developed interesting techniques from harmonic analysis. See Croot and Hart \cite{CH} and the references within for another perspective of the problem.

\section{Main decomposition result}

We recall that the proof of Theorem \ref{bw} \cite{BW} required two ingredients: a way to say if $A$ has additive structure then there is a large subset without multiplicative structure and a simple lemma to understand how multiplicative energy interacts with unions. They then concluded the proof with an iterative argument. We adopt a similar strategy and begin with the former. To begin, we need another definition.

\begin{defn}\label{Dstar}
We define the quantity 
$D^{\times}(A)$ to be the infimum of $$|Q|^2 |R|^2 |A|^{-1} t^{-3},$$ such that  
$$A \subset \{x : r_{Q / R}(x) \geq t\}, \ \ 1 \leq t \leq |Q|^{1/2} |R||A|^{-1/2} , \ \ |R| \leq |Q|.$$ 
We similarly define $D^{+}(A)$ to be the infimum of $|Q|^2 |R|^2 |A|^{-1} t^{-3}$ such that 
$$A \subset  \{x : r_{Q - R}(x) \geq t\} , \ \  1 \leq t \leq |Q|^{1/2} |R||A|^{-1/2} , \ \ |R| \leq |Q| .$$ 
\end{defn}
Thus $D^{\times}(A)$ is small if we can efficiently place $A$ into a set of popular quotients. The admittedly strange quantity $|Q|^2 |R|^2 |A|^{-1} t^{-3}$ is chosen in light of \eqref{key} below.
To understand $D^{\times}(A)$ a bit better, note that taking $Q = AB$, $R = B$ and $t = |B|$ for any $B$ finite, nonempty and not containing zero, one finds
 \begin{equation}\label{ineq}D^{\times}(A) \leq \frac{|AB|^2 }{|A||B|}.\end{equation}
 Thus $D^{\times}(A) \leq |A|$ ($|B| = 1$) and is smaller when $|AA|$ is significantly smaller than $|A|^{3/2}$.

The sole reason for introducing these quantities is the following proposition that relates $D^{\times}(A)$ to $d^+(A)$, as defined in Definition \ref{dplus}.

\begin{prop}[Lemma 13 in \cite{KS2}] Let $A \subset \mathbb{R}$ be finite. Then \begin{equation}\label{key} d^{+}(A) \lesssim D^{\times}(A), \end{equation} $$d^{\times}(A) \lesssim D^{+}(A) .$$
\end{prop}

In \cite{KS2}, they had a slightly different definition of $D^{+}(A), D^{\times}(A)$, replacing the condition $t \lesssim |Q|^{1/2} |R||A|^{-1/2}$ with $|A| \leq |Q|$. The condition we impose is weaker, but the proof of \eqref{key} works line for line as in \cite[Lemma 13]{KS2}. 

To better understand \eqref{key}, one can check that \eqref{triv}, \eqref{ineq} and \eqref{key} together imply Elekes' \cite{El} bound $|A|^{5/2} \lesssim |AA||A+A|$. 

We briefly summarize the proof of \eqref{key} as it plays a crucial role in what lies below. Consider the following sets of points and lines: $$P = Q \times \{x : r_{A - B}(x) \geq \tau\}, \ \ \ L = \{y = \frac{x}{r} - b : r \in R , \ b \in B\}.$$ The number of incidences is at least $$t \tau \#\{ x : r_{A-B}(x) \geq \tau\}.$$ Then \eqref{key} follows from applying Theorem \ref{ST} (Szemer\'edi--Trotter) and a modest calculation. Thus $D^{\times}(A)$ allows us to efficiently transform the equation $y = a-b$ into $y = \frac{q}{r} - b$ which is better suited for Szemer\'edi--Trotter.

There has been a variety of notational choices for these quantities, but we made our choice for the reason that we wanted the quantity with a capital letter to be larger than the one with a lower case letter. Note that \eqref{key} is the only thing we use that relates addition and multiplication and what follows is massaging this inequality for our purposes. We remark that a symmetric version of the following lemma holds with the roles of $d^+$ and $d^{\times}$ reversed.

\begin{lemma}\label{lem1} Let $ T \subset \mathbb{R}$ be a finite, nonempty set. Then there exists a nonempty $A' \subset T$ such that $$d^{\times}(T) d^+(A') \lesssim \frac{|A'|^2}{|T|}, \ \ \ \ \ |A'| \gtrsim d^{\times}(T).$$
\end{lemma}

If one had that $A' = T$, then Lemma \ref{lem1} would immediately imply Theorem \ref{main}, but this is unfortunately too strong to hope for. We first sketch the main idea of the proof. If $d^{\times}(T)$ is large, then there is a large subset $A' \subset T$ with small $D^{\times}(A')$. This is believable as both quantities are defined  via multiplication. Then we finish by applying \eqref{key} to turn this into additive information about $A'$. 
\begin{proof}
By Definition \ref{dplus} of $d^{\times}(T)$, there is a nonempty $B \subset \mathbb{R}$ such that $$d^{\times}(T) \sim \frac{E_3^{\times}(T , B)}{|T| |B|^2}.$$ Then by the definition of $E_3^{\times}(T,B)$ and a dyadic decomposition, there exists a $\Delta \geq 1$ such that $$E_3^{\times}(T, B) \sim |P| \Delta^3,\ \ \ P = \{x : \Delta \leq r_{T / B} (x)  \leq 2 \Delta \}.$$ 
By another dyadic decomposition, there $q \geq 1$ such that 
$$|P| \Delta \sim \sum_{x \in P} r_{T/B}(x) = \sum_{a \in T} r_{B/P^{-1}}(a) \sim |A'| q, \ \ A' = \{ a ' \in T : q \leq r_{B/P^{-1}}(a') \leq 2q\}.$$
From Definition \ref{Dstar}, we then have $$D^{\times}(A') \lesssim |B|^2 |P|^2 q^{-3} |A'|^{-1},$$ provided that $$q \lesssim |A'|^{-1/2} {\rm max} \{ |P| , |B|\}^{1/2} {\rm min}\{|P| , |B|\}.$$
 Since $|A'| q \sim \Delta |P|$, it is enough to verify
 $$|P| \Delta q \lesssim {\rm max} \{ |P| , |B|\} {\rm min}\{|P| , |B|\}^2.$$ 
This follows from $\Delta \leq |B|$ and $q \leq \min\{ |B| , |P|\}$. Then by \eqref{key}, we find $$d^{+}(A') \lesssim D^{\times}(A') \lesssim  |B|^2 |P|^2 q^{-3} |A'|^{-1} \lesssim  |B|^2 |A'|^2 \Delta^{-3} |P|^{-1} \lesssim \frac{|A'|^2}{d^{\times}(T) |T|}.$$
For the second inequality, we use $\Delta^2 q \leq |T| |B|^2$ to obtain $$|A'| \gtrsim |P| \Delta q^{-1} \sim E_3^{\times}(T,B) q^{-1} \Delta^{-2} \gtrsim d^{\times}(T).$$ 
\end{proof}

The referee observed that Lemma~\ref{lem1} is in a similar spirit to the Balog--Szemer\'edi--Gowers theorem \cite[Theorem 2.29]{TV} geared towards the sum--product problem (one should consider $d^{\times}(A) \geq |A| K^{-1}$ for some small $K \geq 1$). The difference is instead of concluding a large susbet with small product set as in Balog--Szemer\'edi--Gowers, we conclude the weaker condition that there is a large subset with no additive structure. Lemma~\ref{lem1} is quantitatively better since we are able to incorporate both addition and multiplication. 

We now move onto the easier ``union lemma." We remark that we avoid an application of H\"{o}lder's inequality, which appears in \cite{Shdstar}.

\begin{lemma}\label{lem2}
Let $A_1 , \ldots , A_K \subset \mathbb{R}$ be finite and disjoint. Then $$d^{+}(\bigcup_{j=1}^K A_j) \leq |\bigcup_{j=1}^K A_j|^{-1} \left( \sum_{j=1}^K d^{+}(A_j)^{1/3} |A_j|^{1/3}\right)^3.$$
Similarly, $$d^{\times}(\bigcup_{j=1}^K A_j) \leq |\bigcup_{j=1}^K A_j|^{-1} \left( \sum_{j=1}^K d^{\times}(A_j)^{1/3} |A_j|^{1/3}\right)^3.$$  
\end{lemma}
\begin{proof}
Let $B \subset \mathbb{R}$ be arbitrary and finite. Then by disjointness and the triangle inequality in $\ell^3(\mathbb{Z})$ , we have 
\begin{align*}E_3^+ (\bigcup_{j=1}^K A_j , B)^{1/3} & = \left(\sum_x r_{\bigcup_{j=1}^K A_j - B}(x)^3\right)^{1/3} =  \left(\sum_x  \left(\sum_{j=1}^K r_{ A_j - B}(x)\right)^3\right)^{1/3}  \\ & \leq \sum_{j=1}^K\left(  \sum_x r_{A_j -B}(x)^3 \right)^{1/3} = \sum_{j=1}^K E_3^+(A_j , B)^{1/3} .\end{align*} 
Recall Definition \ref{dplus} $$ d^{+}(A) = \sup_{B \neq \emptyset} \frac{E_3^+(A,B)}{|A||B|^2}.$$ Since $B$ was arbitrary, we may take the $B$ that maximizes the left hand side of the above equation, after dividing by $|B|^{2/3}$, and use $E_3^+(A_j , B) \leq d^+(A_j) |A_j| |B|^2$ on the right hand side to finish the proof. The proof of the second statement follows line by line to that of the first.
\end{proof}

We now iterate Lemma \ref{lem1} and prove Theorem \ref{main}. Set $A_0 = \emptyset$ and suppose that $A_0 , A_1 , \ldots , A_{j-1}$ have been defined. Put $T = A \setminus (A_0 \cup \ldots \cup A_{j-1})$ and define $A_j$ via Lemma \ref{lem1} as a nonempty set $A_j \subset T$ such that $$d^{\times}(T) d^{+}(A_j) \lesssim |A_j|^2 |T|^{-1}.$$

We continue this process until 

$$|A_1 \cup \ldots \cup A_K| \geq |A|/2.$$

This process must terminate for some finite $K \leq |A|/2$ since the $A_j$ are nonempty and disjoint. Set $Y = A \setminus (A_1 \cup \cdots \cup A_{K-1})$ and $X = A_1 \cup \cdots \cup A_K$. It is clear that $|X| \geq |A|/2$ and $|Y| \geq |A|/2$, otherwise the process would have stopped at step $K-1$. By Lemma \ref{lem1} and the monotonicity of $|A|d^{\times}(A)$, for $1 \leq j \leq K$, 
$$|Y|d^{\times}(Y) d^+(A_j) \leq |A \setminus (A_1 \cup \cdots \cup A_{j-1})|d^{\times}(A \setminus (A_1 \cup \cdots \cup A_{j-1})) d^+(A_j)  \lesssim |A_j|^2.$$ 
Combining with Lemma \ref{lem2},  we have \begin{align*} d^{+}(\bigcup_{j=1}^K A_j) & \leq |\bigcup_{j=1}^K A_j|^{-1} \left( \sum_{j=1}^K d^{+}(A_j)^{1/3} |A_j|^{1/3}\right)^3 \\ & \lesssim |X|^{-1} \left( \sum_{j=1}^K |A_j| \right)^3  |Y|^{-1} d^{\times}(Y)^{-1} \lesssim |X|^2 |A|^{-1} d^{\times}(Y)^{-1}. \end{align*} Theorem \ref{main} follows from $|X| \leq |A|$.

\section{Difference--quotient estimate and Balog--Wooley decomposition}

We start with Theorem \ref{diffquot}. It follows from this stronger proposition.

\begin{prop}\label{dq}
Let $A \subset \mathbb{R}$. Then $$|A-A||AA^{-1}| \gtrsim |A|^{13/5},$$
$$|A-A|^{35}|AA|^{37} \gtrsim |A|^{93}.$$
\end{prop}

\begin{proof}
By Theorem \ref{main}, there exist $X,Y \subset A$ such that $|X|, |Y| \geq |A|/2$ and $$d^+(X)d^{\times}(Y) \lesssim |A|.$$ By the second statement of Theorem \ref{sum}, $$d^+(X) \gtrsim |X|^{8/3} |X-X|^{-5/3} , \ \ \ \ \ d^{\times}(Y) \gtrsim |Y|^{8/3} |YY^{-1}|^{-5/3}.$$ Combining these, we get $$\frac{|A|^{16/3}}{|A-A|^{5/3} |AA^{-1}|^{5/3}} \lesssim \frac{|X|^{8/3}|Y|^{8/3}}{|X-X|^{5/3} |YY^{-1}|^{5/3}} \lesssim |A|.$$ The only difference in the proof of the second statement is we use the first statement of Theorem \ref{sum} in the form $$d^{\times}(Y) \gtrsim |Y|^{58/21} |YY|^{-37 / 21},$$ in place of $d^{\times}(Y) \gtrsim |Y|^{8/3} |YY^{-1}|^{-5/3}$.

\end{proof}

We now prove Theorem \ref{decomp}, which we restate for the reader's convenience. 

\decomp*

The proof can be summarized as iterating Theorem \ref{main} at most logarithmically many times, and then applying the third statement of Theorem \ref{sum} for the first inequality and Cauchy--Schwarz for the second.
\begin{proof}
By Theorem \ref{main}, there exists $A_1$ such that $|A_1| \geq |A|/2$ and $d^+(A_1) \lesssim |A|^{1/2}$ or $d^{\times}(A_1) \lesssim |A|^{1/2}$. Similarly, suppose $A_1, \ldots , A_{K-1}$ are defined. Then by Theorem \ref{main}, there exists $A_K \subset A \setminus (A_1 \cup \cdots \cup A_{K-1})$ such that $|A_K| \geq |A \setminus (A_1 \cup \cdots \cup A_{K-1})|/2$ and $d^+(A_K) \lesssim |A|^{1/2}$ or $d^{\times}(A_K) \lesssim |A|^{1/2}$.

Continue this process until $|A_K| \leq |A|^{1/2}$, since then we trivially have that $d^+(A_K) \leq |A|^{1/2}$. By size considerations, this process will terminate in $\leq \log |A|$ steps.

Let $S \subset \{1 , \ldots , K\}$ be the set of indices $j$ such that $d^+(A_j) \lesssim |A|^{1/2}$ and $P$ be the remaining indices and set $$B = \bigcup_{j \in S} A_j, \ \ \ \ C = \bigcup_{j \in P}A_j.$$  
Then by Lemma \ref{lem2} and H\"{o}lder's inequality, since $|S| \leq \log |A|$ and $d^+(A_j) \lesssim |A|^{1/2}$, we find \begin{align*}d^{+}(B) = d^{+}(\mathop{\cup}_{j \in S} A_j) & \lesssim  |\mathop{\cup}_{j \in S} A_j|^{-1} \left( \sum_{j \in S} |A_j|^{1/3} d^+(A_j)^{1/3} \right)^3 \\ & \lesssim |\mathop{\cup}_{j \in S} A_j|^{-1} \sum_{j \in S} |A_j| d^+(A_j) \lesssim |A|^{1/2} .
\end{align*} 
Similarly $d^{\times}(C) \lesssim |A|^{1/2}$. 

To conclude the first inequality in the second statement, note by the third inequality of Theorem \ref{sum}, we have $$E^{+}(B) \lesssim d^{+}(B)^{\frac{7}{13}} |B|^{\frac{32}{13}} \lesssim |A|^{3 - 7/26},$$ and similarly $E^{\times}(C) \lesssim |A|^{3 - 7/26}$.

For the second inequality in the second statement, we apply Cauchy--Schwarz to obtain
\begin{align*}
E^+(B,A) & \leq E_3^+(B,A)^{1/2} \left(\sum_x r_{A-B}(x) \right)^{1/2} \\
& \leq (d^+(B)|B||A|^2)^{1/2} (|A||B|)^{1/2}\\
& \lesssim |A|^{1/4} |A|^{5/2}.
\end{align*}
Similarly $E^{\times}(C,A) \lesssim |A|^{11/4}$.
\end{proof}

\section{Expander inequalities}

We use Theorem \ref{main} and the techniques of \cite{MRS} to establish the three inequalities of Theorem \ref{expander}. We first recall the two lemmas from their paper that we use.

\begin{lemma}\label{one} [\cite[Lemma 8]{MRS}] Let $A , B \subset \mathbb{R}$ be finite and nonempty such that $|A| \sim |B|$. Then there exists $b \in B$ such that $$|A|^6 \lesssim |A(A+b)|^2 E^{\times}(A).$$
\end{lemma}

\begin{lemma}\label{two}[\cite[Lemma 11]{MRS}] Let $A \subset \mathbb{R}$ be finite and nonempty. Then for all nonzero $\alpha \in \mathbb{R}$,  $$E^{\times}(A)^2\lesssim |A(A+ \alpha)|  |A|^{\frac{58}{13}} d^{+}(A)^{\frac{7}{13}}  .$$
\end{lemma}

Note that the authors only claim Lemma \ref{two} with $D^{\times}(A)$ in place of $d^{+}(A)$ which is weaker in light of \eqref{key}. The bound we claim follows from the same proof, which the authors mention immediately following their proof of Lemma \ref{two}.

\begin{proof}[Proof of Theorem \ref{expander}]
Observe that Lemma \ref{one} is good when the multiplicative energy of $A$ is small and Lemma \ref{two} is good when the multiplicative energy of $A$ is large. We now plan to estimate $$\max_{a \in A} |A(A\pm a)| , |A(A\pm A)|  .$$ We can start all three proofs in the same way. 

\begin{lemma}\label{combine} Let $A\subset \mathbb{R}$ be finite and nonempty. Then there exist $a ,b \in A$ such that $$|A|^{46/13} \lesssim |A(A+a)|^2 d^{\times}(A)^{7/13} ,$$ $$|A|^{46/13} \lesssim |A(A-b)|^2 d^{\times}(A)^{7/13} .$$ Suppose further that $|A(A+a)| \leq r|A|^{3/2}$ for all $a \in A$. Then there is an $X \subset A$ of size at least $|A|/2$ such that  $$d^{+}(X) \lesssim r^{26/7}.$$
\end{lemma}

\begin{proof} By Lemma \ref{one}, there is an $a \in A$ such that $|A|^6 \lesssim |A(A\pm a)|^2 E^{\times}(A)$. Combining this with the third inequality of Theorem \ref{sum} and simplifying yields the first statement of the lemma. To obtain the second statement, let $X$ and $Y$ be as given by Theorem \ref{main}. To finish, apply the first statement to $Y$, use $d^{+}(X)d^{\times}(Y) \lesssim |A|$, and simplify.
\end{proof}

Now we investigate each expander separately.

(i)[$A(A-A) , A(A+A)]$ Suppose $|A(A\pm A)| \leq r |A|^{3/2}$. Since $A(A\pm a) \subset A(A \pm A)$ for all $a \in A$, we may apply Lemma \ref{combine} to obtain a set $X \subset A$ of size at least $|A|/2$ such that $d^+(X) \lesssim r^{26/7}$. Now, using $|A(A\pm A)| \geq |X \pm X|$ along with the first statement of Theorem \ref{sum} in the plus case and the second statement of Theorem \ref{sum} in the minus case and simplifying gives $r \gtrsim |A|^{1/46}$ and $r \gtrsim |A|^{7/226}$, respectively.

(ii)[$A(A\pm a)$] Suppose $$\max_{a \in A} |A(A\pm a)| \leq r |A|^{3/2}.$$ By Lemma \ref{combine}, there is an $X \subset A$ of size at least $|A|/2$ such that $d^+(X) \leq r^{26/7}$. On the other hand, by Lemma \ref{one} and Lemma \ref{two}, $$|A|^3 \lesssim E^{\times}(X) r^2, \ \ E^{\times}(X)^2 \lesssim r |A|^{3/2 + 58/13} d^{+}(X)^{7/13}.$$ Combining these and using $d^{+}(X) \lesssim r^{26/7}$ yields $r \gtrsim|A|^{1/182}$. \end{proof}

\section{Sum--product estimate}

We now proceed to prove Theorem \ref{sumprod}. The proof set--up is the same as in \cite{ KS, KS2}, which we now discuss. Let $A \subset \mathbb{R}$ be finite. Konyagin and Shkredov start with the geometric approach of Solymosi \cite{So}, and can improve upon it unless $A_{\lambda} : = A \cap \lambda A$ has additive structure for many choices of $\lambda$. They then prove an energy analog of a ``few sums, many products" result in \cite{ER} and use it to conclude that $|A_{\lambda} A_{\lambda}|$ is almost as big as possible. It turns out that these sets are relatively small  ($\approx |A|^{2/3}$) and this does not immediately improve Solymosi's \cite{So} exponent of 1/3 in Conjecture \ref{esconj}. Konyagin and Shkredov then use Katz--Koester \cite{KK} inclusion, that is $$A_{\lambda} A_{\lambda} \subset AA \cap \lambda AA,$$ as well as \eqref{key} and the first inequality of Theorem \ref{sum}, to also give an improvement in this case. In what remains, we quantitatively improve part of the argument and provide the entirety of the proof of \cite{KS2} to see how our new pieces fit in.

The ``few sums, many products" lemma was improved recently in \cite{RSS}. We interpret this improvement as a fourth order energy estimate which allows us to more efficiently apply the lemma. The work in \cite{RSS} relied on bounding the number of solutions to $$ \frac{p+b}{q + c} = \frac{p' + b'}{q' + c'} , \ \ \ p , q , p' , q' \in P  , \ \  b , c , b' , c' \in B,$$ which was addressed \cite{MRS1} while studying the expanders from Theorem \ref{expander}. It turns out, much like Szemer\'edi--Trotter is naturally a third moment estimate, their lemma is naturally a fourth moment estimate.

We use fourth order energy for the first time in the sum--product problem and define $$ E_4^{+}(A,B) = \sum_x r_{A-B}(x)^4 , \ \ \ \ \ \ \ \ E_4^{\times}(A,B) = \sum_x r_{A/B}(x)^4.$$ 
Similar to $d^{+}(A)$ as in Definition \ref{dplus}, we define $d_4^+(A)$.

\begin{defn}\label{d4} Let $A \subset \mathbb{R}$ be finite. Then we define $d_4^+(A)$ via $$d_4^+(A) := \sup_{B \neq \emptyset} \frac{E_4^+(A,B)}{ |A| |B|^3}.$$ \end{defn}

It is easy to see that one has $1 \leq d_4^+(A) \leq |A|$ and in fact we have $d_4^+(A) \leq d^+(A)$. So $d_4^+(A)$ is the fourth moment analog of $d^+(A)$. Note that the supremum in Definition \ref{d4} is obtained for some $|B| \leq |A|^{3/2}$, since $$\frac{E_4^+(A, B)}{|A| |B|^3} \leq \frac{|A|^3}{|B|^2}.$$Similar to Remark \ref{op}, we relate $d_4^+(A)$ to an operator norm. 
\begin{remark}
Consider the linear operator, as in Remark \ref{op},
\begin{equation}\label{operator} T_A (f):= \sum_x 1_A(x) f(y+x).\end{equation}
Then 
$$d_4^+(A) = \frac{1}{|A|} ||T_A||_{\ell^{4/3} \to \ell^4}^4 , \ \ \ d^+(A) =  \frac{1}{|A|} ||T_A||_{\ell^{3/2} \to \ell^3}^3 .$$
It is easy to see that $$||T_A||_{\ell_1 \to \ell_{\infty}} = |A| , \ \ \ |A|^{1/2} \leq ||T_A||_{\ell_2 \to \ell_2} \leq |A|.$$ A bound of the form $$||T_A||_{\ell_2 \to \ell_2} \lesssim |A|^{1/2},$$ together with interpolation with $\ell_1 \to \ell_{\infty}$ implies $d_4^+(A) \leq d^+(A) \lesssim 1$. Thus $\ell_2 \to \ell_2$ estimates are stronger, but higher moments are flexible to work with, as in Theorem \ref{main} above and Proposition \ref{second} below. 
\end{remark}

We now need the following quantity, which plays an important role in the Konyagin--Shkredov argument.
\begin{defn}
Let $A, B , C \subset \mathbb{R}$ be finite and define
$$\sigma(A,B,C) := \sup_{\sigma_1, \sigma_2 , \sigma_3 \neq 0} \#\{(a, b , c) \in A \times B \times C : \sigma_1a + \sigma_2 b + \sigma_3 c = 0\}.$$
\end{defn}

We have the trivial bound $\sigma (A , B , C) \leq |A| |B|$ and this is basically obtained when $A, B,C = \{1 , \ldots , n\}$. We expect that $\sigma (A, B , C)$ is small whenever $A, B$, or $C$ has little additive structure. Konyagin and Shkredov \cite{KS2} used that $$\sigma(A , B , C) \leq |A|^{1/2} E^+(B)^{1/4} E^+(C)^{1/4},$$ which we replace with the following.

\begin{prop}\label{first}
Let $A, B ,C \subset \mathbb{R}$ be finite. Then $$\sigma(A,B,C) \leq |C|^{3/4} \left( d_4^+(A) |A| |B|^3  \right)^{1/4}$$
\end{prop}

\begin{proof}
The proof is similar to what appears in \cite{LR} for third order energy. Fix $\sigma_1 , \sigma_2 , \sigma_3 \neq 0$. Then by H\"older's inequality, we obtain

\begin{align*}
\#\{(a, b , c) \in A \times B \times C : &\sigma_1a + \sigma_2 b + \sigma_3 c = 0\}  = \sum_{c \in C} \#\{(a, b ) \in A \times B  : \sigma_1a + \sigma_2 b = -\sigma_3 c \} \\
& \leq |C|^{3/4} \left( \sum_{c \in C} \#\{(a, b ) \in A \times B  : \sigma_1a + \sigma_2 b = -\sigma_3 c \}^4 \right)^{1/4} \\
& \leq |C|^{3/4} \left( \sum_{x} \#\{(a, b ) \in A \times B  : a + \sigma_2 / \sigma_1 b = x  \}^4 \right)^{1/4} \\
& \leq |C|^{3/4} \left( d_4^+(A) |A| |B|^3  \right)^{1/4} .\\
\end{align*}
The proposition follows as $\sigma_1 , \sigma_2 , \sigma_3$ were arbitrary.
\end{proof}

The next lemma is a fourth order energy analog of \cite[Theorem 12]{RSS}. 

\begin{prop}  \label{second} Let $A \subset \mathbb{R}$ be finite. Then there exists $X , Y \subset A$ such that
\begin{itemize}
\item[(i)] $X \cup Y = A$,
\item[(ii)] $|X| , |Y| \geq |A| /2$,
\item[(iii)] $d_4^+(X) E^{\times}(Y) \lesssim |A|^3$.
\end{itemize}
\end{prop}

Proposition \ref{second} is a ``few sums, many products" theorem. Indeed, if $d_4^+(X) \gtrsim |A|$, then $E^{\times}(Y) \lesssim |A|^2$ and so by Cauchy--Schwarz, $$|AA| \geq |YY| \gtrsim |Y|^2.$$ We begin the proof of Proposition \ref{second} with the following lemma. We mention that there is a large overlap of the proof of Theorem \ref{main} and Proposition \ref{second}, which are both decomposition results.

\begin{lemma}\label{lem11}Let $A \subset \mathbb{R}$ be finite. Then there exists a nonempty $A' \subset A$ such that
$$E^{\times}(A')   d_4^+(A)  \lesssim \frac{|A'|^4}{|A|} , \ \ \ |A'| \gtrsim d_4^{+}(A).$$
\end{lemma}

Note that if $A'$ were equal to $A$ then Proposition \ref{second} would immediately follow, but this is too strong to hope for.

\begin{proof}
Let $B \subset \mathbb{R}$ be finite and nonempty. By a dyadic decomposition, there is a $\Delta \geq 1$ such that $$E_4^+(A,B) \sim |P| \Delta^4, \ \ \ P = \{x : \Delta \leq r_{A-B}(x) \leq 2 \Delta\}.$$ We double count the number of solutions to \begin{equation}\label{sols} \frac{p+b}{q + c} = \frac{p' + b'}{q' + c'} , \ \ \ p , q , p' , q'  \in P  , \ \  b , c , b' , c' \in B.\end{equation} By a claim in the proof of \cite[Lemma 2.5]{MRS1}, one has that the number of solutions to \eqref{sols} is $\lesssim |P|^3 |B|^3$.

By a dyadic decomposition, there is a $q \geq 1$ such that $$|A'|q \sim \sum_{a \in A} r_{P+B}(a) \sim \sum_{x \in P} r_{A-B}(x) \sim \Delta |P|, \ \ A' = \{a' \in A : q \leq r_{B+P}(a')   < 2q\}.$$  
Given $$\frac{a_1}{a_2} = \frac{a_3}{a_4} , \ \ \ \ a_1 , a_2 , a_3 , a_4 \in A',$$ we may create a solution to \eqref{sols}, via $$\frac{a_1 - b_1 + b_1}{a_2 - b_2 + b_2} = \frac{a_3-b_3 + b_3}{a_4-b_4 + b_4} , \ \ \ \ b_1 , b_2 , b_3 , b_4 \in B, $$ as long as $a_j - b_j \in P$ for all $j$. Since each $a_j \in A'$, there are at least $q$ such choices for each $b_j$. Thus $q^4 E^{\times}(A')$ is $\lesssim$ the number of solutions to \eqref{sols}, and so
 $$E^{\times}(A') \sim \frac{|B|^3 |P|^3}{q^4} \sim \frac{|A'|^4|B|^3 }{|P| \Delta^4}   \sim \frac{|A'|^4 |B|^3 }{E_4^+(A,B)} .$$ 
Finally, using $\Delta \leq |B|$ and $q \leq |A|$, we have $$|A'| \gtrsim |P| \Delta q^{-1} \gtrsim E_4^+(A,B) |A|^{-1} |B|^{-3}.$$
The lemma now follows from Definition \ref{d4} of $d_4^+(A)$ since $B$ is arbitrary. \end{proof}
We also need the following lemma describing how $E^{\times}(A)$ behaves with respect to unions. The lemma will require the following application of Cauchy--Schwarz \begin{equation}\label{cs} E^{\times}(A,B)^2 \leq E^{\times}(A) E^{\times}(B). \end{equation}

\begin{lemma}\label{lem3}
Let $A_1 , \ldots , A_K \subset \mathbb{R}$ be finite and disjoint. Then $$E^{\times}(\bigcup_{j=1}^K A_j) \leq  \left( \sum_{j=1}^K E^{\times}(A_j)^{1/4}\right)^4$$

\end{lemma}

\begin{proof}
By the triangle inequality in $\ell^2(\mathbb{Z})$, we have \begin{align*}E^{\times} (\bigcup_{j=1}^K A_j )^{1/2} & = \left(\sum_x r_{\bigcup_{j=1}^K A_j / \bigcup_{k=1}^K A_k }(x) ^2\right)^{1/2} =  \left(\sum_x  \left( \sum_{j, k=1}^K r_{ A_j / A_k}(x) \right)^2\right)^{1/2}  \\ & \leq \sum_{j,k=1}^K\left(  \sum_x r_{A_j /A_k}(x)^2 \right)^{1/2} = \sum_{j,k=1}^K E^{\times}(A_j , A_k)^{1/2} \end{align*} Now we apply \eqref{cs} to obtain $$\sum_{j,k=1}^K E^{\times}(A_j , A_k)^{1/2}  \leq \sum_{j,k=1}^K E^{\times}(A_j)^{1/4} E^{\times} ( A_k)^{1/4} = \left( \sum_{j=1}^K E^{\times}(A_j )^{1/4} \right)^2.$$ Combining these two inequalities completes the proof. 
\end{proof}
We now iterate Lemma \ref{lem11} and prove Proposition \ref{second}.

\begin{proof}[Proof of Proposition \ref{second}] Set $A_0 = \emptyset$ and suppose that $A_0 , A_1 , \ldots , A_{j-1}$ have been defined.  We define $A_j$ via Lemma \ref{lem11} as a nonempty set $A_j \subset A \setminus (A_0 \cup \ldots \cup A_{j-1})$ such that $$d_4^+(A\setminus (A_0 \cup \cdots \cup A_{j-1})) E^{\times}(A_j) \lesssim |A_j|^4 |A \setminus (A_0 \cup \ldots \cup A_{j-1})|^{-1}.$$ We continue this process until 
$$|A_1 \cup \ldots \cup A_K| \geq |A|/2.$$
This process must terminate for some $K \leq |A|/2$ as the $A_j$ are nonempty and disjoint. 

Set $X = A \setminus (A_1 \cup \cdots \cup A_{K-1})$ and $Y = A_1 \cup \cdots \cup A_K$. It is clear that $|Y| \geq |A|/2$ and $|X| \geq |A|/2$, otherwise the process would have stopped at step $K-1$. By Lemma \ref{lem11}, for $1 \leq j \leq K$, 
$$d_4^+(X) E^{\times}(A_j) \leq d_4^+(A \setminus (A_1 \cup \cdots \cup A_{j-1})) E^{\times}(A_j)  \lesssim |A_j|^4 |A \setminus ( A_1 \cup \cdots \cup A_{j-1})|^{-1} \lesssim |A_j|^4 |A|^{-1}.$$ Combining this with Lemma \ref{lem3},  we obtain
\begin{align*} E^{\times}(Y) &= E^{\times}(\bigcup_{j=1}^K A_j)  \leq \left( \sum_{j=1}^K E^{\times}(A_j)^{1/4} \right)^4 \\ & \lesssim \frac{1}{d_4^+(X)|A|}\left( \sum_{j=1}^K |A_j| \right)^4  \leq |Y|^3 d_4^+(X)^{-1}. \end{align*} Proposition \ref{second} follows from $|Y| \leq |A|$.

\end{proof}

We now give the proof of Theorem \ref{sumprod}, which is identical to that in \cite{KS2} with some minor changes to utilize Proposition \ref{first} and Proposition \ref{second}. 

\begin{proof}[Proof of Theorem \ref{sumprod}] Suppose that $|A+A| , |AA| \leq r|A|^{4/3}$. Thus our goal is to show $$r \gtrsim |A|^{5/5277}.$$ 
Note that by Cauchy--Schwarz, $|AA| \leq r |A|^{4/3}$ implies \begin{equation}\label{ener} E^{\times}(A) \geq |A|^{8/3} r^{-1}.\end{equation}
 By a dyadic decomposition, there exists a $$t \geq E^{\times}(A) |A|^{-2} \geq |A|^{2/3} r^{-1},$$ such that $$E^{\times}(A) \sim |S_t| t^2 , \ \ S_t = \{\lambda : r_{A/A}(\lambda) \sim t \}.$$ 
 For $\lambda \in A/A$, we set $$A_{\lambda} = A \cap \lambda A.$$ Thus $|A_{\lambda} | = r_{A/A}(\lambda )$.
By Konyagin--Shkredov clustering   \cite{KS2} (see also Adam Sheffer's blog \cite[Equation 9]{Shef} or my blog \cite{GS}), which is a refinement of Solymosi's \cite{So} argument, we have  \begin{equation}\label{cluster}|A+A|^2 \gtrsim \frac{|S_t|}{M} \left( M^2 t^2 - M^4 \max_{\lambda_1, \lambda_2 , \lambda_3 \in S_t} \sigma(A_{\lambda_1} , A_{\lambda_2} , A_{\lambda_3}) \right),\end{equation} as long as $2 \leq M \leq |S_t| /2$. We apply Proposition \ref{first} to obtain
$$|A+A|^2 \gtrsim \frac{|S_t|}{M} \left( M^2 t^2 - M^4t^{7/4} \max_{\lambda \in S_t} d_4^+(A_{\lambda})^{1/4} \right).$$
Set $$M^2 := \frac{t^{1/4}}{2 \max_{\lambda \in S_t} d_4^+(A_{\lambda})^{1/4}}.$$
Now, we have $M \leq |S_t| / 2$, since otherwise, using $d^+(A_{\lambda}) \geq1$, $$|A|^{4 - 4/3} r^{-1} \leq E^{\times}(A) \sim |S_t| t^2 \lesssim M t^2  \lesssim t^{17/8} \leq |A|^{17/8} \leq |A|^{2 + 1/3},$$ and so $r \gtrsim |A|^{1/3}$. Also, if $M \geq 2$, then we may apply \eqref{cluster} to obtain $$ME^{\times}(A) \lesssim |A+A|^2,$$ which implies $M \lesssim r^3$. Note that Solymosi originally proved $E^{\times}(A) \lesssim |A+A|^2$. We can improve unless $M$ is very small.

Thus we just have to handle the hardest case: $M \lesssim r^{3}$, that is $$\frac{t}{r^{24}} \leq d_4^+(A).$$ By a technical trick in \cite{KS2}, this implies for all $\lambda \in S_t$ (as opposed to maximum in $\lambda$), that \begin{equation}\label{d4ineq} \frac{t}{r^{24}} \lesssim d_4^+(A_{\lambda}).\end{equation} Indeed, we may partition $S_t = S_t' \cup S_t''$ where $d_4^+(A_{\lambda'}) \leq d_4^+(A_{\lambda''})$ for $\lambda' \in S_t'$ and $\lambda'' \in S_t''$. Then we apply the above argument to $S_t'$ and see that \eqref{d4ineq} holds for all the elements in $S_t''$. Then we may replace $S_t$ with $S_t''$ at the loss of just a constant. Note that \eqref{d4ineq} implies that $d_4^+(A_{\lambda})$ is almost as large as possible and each $A_{\lambda}$ has a lot of additive structure. 

After passing to large subsets of $A_{\lambda}$ and applying Proposition \ref{second}, we have $$E^{\times}(A_{\lambda}) \lesssim t^2 r^{24}, \  \ \ \lambda \in S_t,$$ and by Cauchy--Schwarz we find $$\frac{t^2}{r^{24}} \lesssim |A_{\lambda} A_{\lambda}|, \ \ \ \lambda \in S_t.$$ Thus each $A_{\lambda}$ has almost no multiplicative structure. 
By  Katz--Koester inclusion \cite{KK}, we have $A_{\lambda} A_{\lambda} \subset AA \cap \lambda AA$ and so $$S_t \subset \{x : r_{AA/ AA}(x) \gtrsim t^2 r^{-24} \}.$$ One issue is that $S_t$ is a subset of $AA^{-1}$ and not $A$. By a popularity argument, since $$|S_t| t \leq \sum_{\lambda \in S_t} |A \cap \lambda A| = \sum_{a \in A} |A \cap a S_t|,$$ there is an $a \in A$ such that $$A' = a S_t \cap A \subset \{ x : r_{a AA / AA}(x) \gtrsim t^2 r^{-24} \}, \ \ \ |A'| \gtrsim |S_t| t |A|^{-1}.$$ Thus by \eqref{key}, we find $$d^{+}(A') \leq D^{\times}(A') \lesssim \frac{|AA|^4 r^{72}}{|A'| t^6}.$$ We now apply the first statement of Theorem \ref{sum} (To prove Theorem \ref{random}, one should apply the second statement of Theorem \ref{sum} in place of the first) and first use that $ |A'| \gtrsim |S_t| t |A|^{-1}$ and $|S_t|t^2 \sim E^{\times}(A)$ to obtain
 $$|A+A|^{37} \geq |A'+ A'|^{37} \gtrsim \frac{E^{\times}(A)^{79} t^{47}}{|AA|^{84} |A|^{79}r^{1512}}.$$ We now apply $t \gtrsim E^{\times}(A) |A|^{-2}$ and then \eqref{ener} to find 
$$|A|^{331} \lesssim r^{1512} |A+A|^{37} |AA|^{210}.$$
 Theorem \ref{sumprod} then follows from $|AA| , |A+A| \leq r|A|^{4/3}$ and simplification. 

\end{proof}

\end{document}